\documentclass{amsart}
\usepackage[english]{babel}
\usepackage[utf8]{inputenc}
\usepackage[autostyle]{csquotes}

\usepackage{latexsym}
\usepackage{amsfonts}
\usepackage{euscript}
\usepackage{times}

\usepackage{dsfont}
\usepackage{amsmath}
\usepackage{enumitem}
\usepackage{amsthm, amssymb}
\usepackage{bm}
\usepackage{url}

\usepackage{booktabs}
\usepackage{caption}

\usepackage{imakeidx}
\makeindex[columns=1]

\theoremstyle{plain} 
\newtheorem{teo}{Theorem}[section]
\theoremstyle{plain} 
\newtheorem{teor}{Theorem}[section]

\theoremstyle{definition}

\theoremstyle{plain} 
\newtheorem{prop}[teo]{Proposition}
\theoremstyle{plain}
\newtheorem{lem}[teo]{Lemma}
\theoremstyle{plain}
\newtheorem{cor}[teo]{Corollary}	
\theoremstyle{definition}
\newtheorem{oss}[teo]{Remark}
\theoremstyle{definition}

\theoremstyle{definition}

\theoremstyle{plain}

\newcommand*{\sn}{\unlhd \unlhd \ }
\DeclareMathOperator{\spr}{sp}
\DeclareMathOperator{\fpr}{fpr}

\DeclareMathOperator{\fix}{Fix}

\begin{document}

\title{Subnormalizers and solvability in finite groups}
\author{Pietro Gheri}
\address{Dipartimento di Matematica e Informatica ``U. Dini'',\newline
Universit\`a degli Studi di Firenze, viale Morgagni 67/a,
50134 Firenze, Italy.}
\email{pietro.gheri@unifi.it}

\dedicatory{This work is dedicated to the memory of Carlo Casolo.\\
His knowledge, his curiosity, his humility and his humanity were an example to all of his students and friends.}

\begin{abstract}
For a finite group $G$, we study the probability $\spr(G)$ that, given two elements $x,y \in G$, the cyclic subgroup $\langle x \rangle$ is subnormal in the subgroup $\langle x, y \rangle$. This can be seen as an intermediate invariant between the probability that two elements generate a nilpotent subgroup and the probability that two elements generate a solvable subgroup. We prove that $\spr(G) \leq 1/6$ for every nonsolvable group $G$. 
\end{abstract}

\maketitle

\section{Introduction}

Let $G$ be a finite group. Given a group theorethical property $\mathcal{X}$ one can define the \textit{degree of $\mathcal{X}$ of $G$} as the probability that two randomly chosen elements in $G$ generate a $\mathcal{X}$-subgroup, that is
\[
\frac{\lbrace (x,y) \in G \times G \ | \ \langle x, y \rangle \mbox{ is a } \mathcal{X}-\mbox{group} \rbrace }{|G|^2}.
\]  

What seems a necessary condition for this probability to be meaningful is that $\mathcal{X}$ is closed under subgroups and quotient groups. 

The degree of commutativity was widely studied since W.H. Gustafson first introduced it in \cite{gustafson:degcom}, proving that the degree of commutativity of a nonabelian finite group is at most $5/8$. The degree of nilpotence and of solvability were the the focus of \cite{guralnick:solvable}, where two Gustafson-like results were proved for these probabilities: the degree of nilpotence of a nonnilpotent group is at most $1/2$ and that of solvability of a nonsolvable group is at most $11/30$. All these bounds are best possible.

In \cite{gheri:degnil}, in order to treat the degree of nilpotence of $G$, two linked quantities turned out to be useful, namely the ratio $\spr_G(x)$ and the probability $\spr(G)$ defined respectively by
\begin{equation} \label{sprgx}
\spr_G(x)=\dfrac{|S_G(x)|}{|G|}
\end{equation}
\begin{equation} \label{sprg}
\spr(G)= \frac{1}{|G|}\sum_{g \in G} \spr_G(x),
\end{equation}
where $S_G(x)$ is the Wielandt's subnormalizer of the element $x$ of $G$, that is, 
\[
S_G(x)= \lbrace g \in G \ | \ \langle x \rangle \sn \langle x,g \rangle \rbrace,
\]
(where $\sn$ means ``is subnormal in").

The probability $\spr (G)$ is clearly at most the degree of nilpotence and at least the degree of solvability. Therefore, if $\spr(G) > 11/30$, then $G$ is solvable. However the bound $11/30$ is not best possible, as one can expect. 

In this paper we prove the following theorem.

\begin{teor} \label{unsesto}
Let $G$ be a finite group. If $\spr (G)> 1/6$, then $G$ is solvable. Moreover, the bound is best possible.
\end{teor}

The proof relies on CFSG, while the fact that the bound is best possible follows from a direct calculation that gives $\spr(A_5)=1/6$. 

A huge part of the proof of Theorem \ref{unsesto}  consists in using known bounds on the so-called \textit{fixed point ratio}. Given an element $x$ of a finite group $G$ acting on a set $\Omega$, the fixed point ratio of $x$ is the proportion of the elements of $\Omega$ that are fixed by $x$. The fixed point ratio has been widely studied (see for example \cite{burness:fprclassicalII}, \cite{liebeckshalev:simplegpspermgps}, \cite{gluckmagaard:charfprcclas} and the survey by T.C. Burness \cite{burness:surveyfpr}) and the reason why it is useful in our case is that when $x$ is a $p$-element for some prime $p$, then $\spr_G(x)$ is the fixed point ratio of $x$ with respect to the action of $G$ on its Sylow $p$-subgroups.

All groups considered in the paper are finite.

\section{Outline of the proof} \label{sketch}

In this section we describe the main steps of the proof of Theorem \ref{unsesto}, whose details we give in next sections.

Let $G$ be a nonsolvable group. We want to prove that $\spr(G) \leq 1/6$.
We state the following lemma, whose proof is straightforward.

\begin{lem}
Let $N$ be a normal subgroup of $G$. Then $\spr (G/N) \geq \spr (G)$.
\end{lem}

Arguing by induction, we can thus assume that $G$ is nonsolvable, while all its proper quotients are solvable. In particular, $G$ admits a unique minimal normal subgroup $N$, which is nonabelian.

A group satisfying this property is called a \textit{minimal nonsolvable monolithic group} and we set $\mathfrak{M}_{ns}$ to be the class of these groups. Then we are assuming that $G \in \mathfrak{M}_{ns}$.

To sketch the proof, we introduce a combinatorial meaning of $\spr_G(x)$, when $x$ is a $p$-element, for $p$ a prime dividing $|G|$. In this case, the following nice formula proved by C.~Casolo, concerning the order of subnormalizers of $p$-subgroups, holds. 

\begin{teo}[\cite{casolo:subnor}] \label{form subnor theorem}  Let $H$ be a $p$-subgroup of $G$ and $P \in Syl_p(G)$. Set $\lambda_G(H)$ the number of Sylow $p$-subgroups of $G$ containing $H$ and $\alpha_G(H)$ the number of $G$-conjugates of $H$ contained in $P$. Then
\begin{equation} \label{formulasubnor}
|S_G(H)|=\lambda_G(H) |N_G(P)| = \alpha_G(H)|N_G(H)|.
\end{equation}
\end{teo}

It is then clear, that $\spr_G(x)$ is the proportion of Sylow $p$-subgroups of $G$ containing $x$, namely
\[
\spr_G(x)=  \frac{\lambda_G(x) |N_G(P)|}{|G|}= \frac{\lambda_G(x)}{n_p(G)},
\]
where $\lambda_G(x)=\lambda_G(\langle x \rangle)$ and $n_p(G)$ is the number of Sylow $p$-subgroups of $G$.

Using this fact, in \cite{gheri:degnil} the next proposition is proven.

\begin{prop} \label{uno su p piu uno}
Let $p$ be a prime dividing the order of $G$ and $x$ be a $p$-element of $G$ of order $p^r$. If $\spr_G(x)>1/(p^k+1)$ for some positive integer  $k$, $1 \leq k \leq r$, then $x^{p^{k-1}} \in O_p(G)$.
\end{prop}	

Since in our assumptions $G$ has trivial Fitting subgroup we can assume that 
\[
\spr_G(x) \leq \dfrac{1}{q+1},
\]
where $q$ is the maximum prime power dividing $|x|$. Therefore, if the order of an element $x \in G$ is divided by a prime power $q \geq 5$, its contribution in the sum (\ref{sprg}) is at most $1/6$.	

We denote with $\mathfrak{U}_2(G)$ the set of $2$-elements in $G$, that is, the union of all $2$-Sylow subgroups of $G$. Moreover we let $\mathfrak{V}$ be the set of $\lbrace 2,3 \rbrace$-elements $x$ in $G$ such that $|x|=2^n\cdot 3$ for some $n \in \mathbb{N}$. For such an element $x$, set $x_3 := x^{2^n}$. Note that for every $x \in \mathfrak{V}$, we have $\spr_G(x) \leq \spr_G(x_3)$. As we said above, if $x \in G$ is not in $\mathfrak{V} \cup \mathfrak{U}_2(G)$, then $\spr_G(x) \leq 1/6$.

We can bound the probability $\spr(G)$ as follows:
\begin{align} \label{arg gen un sesto}
 \spr(G) &  =  \frac{1}{|G|}  \sum_{g \in G} \spr_G(g)  \nonumber \\ 
&  \leq  \frac{1}{|G|} \left( \sum_{g \in \mathfrak{U}_2(G)} \spr_G(g) \right) + \frac{1}{|G|} \left( \sum_{g \in \mathfrak{V}} \spr_G(g) \right) + \frac{|G \setminus \left(  \mathfrak{U}_2(G) \cup \mathfrak{V} \right)|}{6|G|} \\
& \leq  \frac{1}{|G|} \left( \sum_{g \in \mathfrak{U}_2(G)} \spr_G(g) \right) + \frac{1}{|G|} \left( \sum_{g \in \mathfrak{V}} \spr_G(g_3) \right) + \frac{|G \setminus \left(  \mathfrak{U}_2(G) \cup \mathfrak{V} \right)|}{6|G|}. \nonumber
\end{align}

We divide the rest of the argument in two steps.

\begin{itemize} 
\item[\textit{Step 1}] In order to bound the middle term, $\frac{1}{|G|} \left( \sum_{g \in \mathfrak{V}} \spr_G(g_3) \right)$, we prove (Proposition \ref{tre un sesto}) that if $G$ nonsolvable group and $x$ is an element of order $3$ of $G$, not lying in the solvable radical of $G$, then
\[
\spr_G(x) \leq \frac{1}{6}.
\]
\end{itemize}

Given \textit{Step 1} we can rewrite (\ref{arg gen un sesto}) as follows:
\begin{equation}
\spr(G) \leq \frac{1}{|G|} \left( \sum_{g \in \mathfrak{U}_2(G)} \spr_G(g) \right) + \frac{|G \setminus \left(  \mathfrak{U}_2(G) \right)|}{6|G|}. 
\end{equation}

\begin{itemize}
\item[\textit{Step 2}] As an easy consequence of Theorem \ref{form subnor theorem}, we will show (Lemma \ref{somma dei subnormalizzatori dei p-elementi}) that the sum in the above formula is just the ratio between the cardinality of a Sylow $2$-subgroup of $G$ and that of $G$.
We then complete the proof of Theorem \ref{unsesto} by proving the following statement
\begin{equation} \label{Frob ratio un sesto}
\frac{| \mathfrak{U}_2(G)|}{|P|} \geq 6.
\end{equation}
We will show this for ``almost" every nonsolvable group $G$. The exceptions will be treated separately.
\end{itemize}

\section{Elements of order $3$}

In this section we prove \textit{Step 1} of the proof. Namely we want to prove

\begin{prop} \label{tre un sesto}
Let $G$ be a finite nonsolvable group with no normal solvable subgroups and let $x$ be an element of $G$ of order $3$.
Then
\[
\spr_G(x) \leq \frac{1}{6}.
\] 
\end{prop}

Before proving Proposition \ref{tre un sesto}, we make a remark about Theorem \ref{form subnor theorem} which will be useful in the rest of the paper. 

\begin{oss} \label{ossalfa}
If $H= \langle x \rangle$ is a cyclic $p$-subgroup of $G$ then, calling $\alpha_G(x)$ the number of $G$-conjugates of $x$ that are contained in a fixed Sylow $p$-subgroup, we have
\[
S_G(x)=\alpha_G(\langle x \rangle ) |N_G(\langle x \rangle )| = \alpha_G(x) |C_G(x)|.
\]
In particular,
\[
\spr_G(x)= \frac{\alpha_G(x) |C_G(x)|}{|G|} = \frac{\alpha_G(x)}{|x^G|}.
\]
\end{oss}

\begin{lem} \label{sp x cresce sui sottogruppi}
Let $p$ be a prime dividing $|G|$, $H$ a subgroup of $G$ and $x \in H$ be a $p$-element. Then the following holds.
\begin{itemize}
\item[i)] $\spr_G(x) \leq \spr_H(x)$.
\item[ii)] If $H \unlhd G$, then $\spr_G(x) = \spr_H(x)$.
\end{itemize}
\end{lem}
\begin{proof}
Let $\mathcal{S}=Syl_p(G)$ be the set of the Sylow $p$-subgroups of $G$ and let
\[
\Lambda = \lbrace (y,P) \in x^H \times \mathcal{S} \ | \ y \in P \rbrace.
\]
We count the elements in $\Lambda$ first by summing on the first component, and then on the second one.
Using Theorem \ref{form subnor theorem} and the fact that $\lambda_G$ is invariant on the conjugacy classes, we get
\[
|\Lambda|= \sum_{y \in x^H} \lambda_G(y) = \lambda_G(x) |x^H|.
\]
Also
\begin{equation} \label{Lambda}
|\Lambda|= \sum_{P \in \mathcal{S}} |x^H \cap P| = \sum_{P \in \mathcal{S}} |x^H \cap (P \cap H)| \leq \sum_{P \in \mathcal{S}} \alpha_H(x) = \alpha_H(x) |\mathcal{S}|,
\end{equation}
since $|x^H \cap (P \cap H)| \leq |x^H \cap Q|$, where $Q$ is a Sylow $p$-subgroup of $H$ containing $P \cap H$.
Finally we have that
\[
\frac{\lambda_G(x)}{|\mathcal{S}|} \leq \frac{\alpha_H(x)}{|x^H|},
\]
which is part (i), by Remark \ref{ossalfa}. 

Part (ii) follows by noticing that the inequality in (\ref{Lambda}) is an equality when $H \unlhd G$.
\end{proof}

\begin{lem} \label{riduzione tre un sesto CFSG}
Proposition \ref{tre un sesto} holds if and only if it holds for $G$ in the following three cases:
\begin{enumerate} 
\item $G$ is a nonabelian simple group. 
\item $G=L\langle x \rangle$, with $L$ a nonabelian simple group and $x \in Aut(L) \setminus L$.
\item $G=(L \times L \times L)\langle x \rangle$, $L$ a nonabelian simple group and $x$ an element that permutes the three factors.
\end{enumerate}
\end{lem}
\begin{proof}
Let $G$ and $x \in G$, $|x|=3$ be a counterexample to Proposition \ref{tre un sesto}, with $G$ of minimal order. Then $\spr_G(x) > 1/6$. Let $N$ be the minimal normal subgroup of $G$. Being nonsolvable, $N$ is a direct product of $k$ copies of a simple group $L$:
\[
N = L_1 \times \dots \times L_k, \ L_i \simeq L, \ \forall i \in \lbrace 1, \dots , k \rbrace.
\]
First suppose that $x \in N$. Then $x=(x_1, \dots, x_k)$, with $|x_i| \in \lbrace 1,3 \rbrace$. Without loss of generality, we can suppose that $x_1 \neq 1$. Let
\[
H=L_1 \times \langle x_2 \rangle \times \dots \times \langle x_k \rangle.
\]
so that $x \in H$.
By Lemma \ref{sp x cresce sui sottogruppi}, $\spr_G(x)$ is smaller than $\spr_H(x)$, which is equal to $\spr_{L_1}(x_1)$.

Suppose now that $x \notin N$. We can assume $L_1 \nleq C_N(x)$. If $x \in N_G(L_1)$, then we can take $H = L_1\langle x \rangle$ and we are in case 2. Otherwise, for every $a \in N$,
\[
((L_1)^x)^a = (L_1)^{a^{x^{-1}}x}=(L_1)^x,
\]
so that $L_1^x$ is a normal subgroup of $N$. We know that the only normal subgroups of a direct product of simple groups are the products of the factors and so $L_1^x$ must be one of the subgroups $L_i$, say $L_2$. If $x \in N_G(L_2)$ then $L_1^{x^2}=L_1^x$ contradicting $x \notin N_G(L_1)$. Also $L_2^x \neq L_1$, for otherwise 
\[
L_1=(L_1)^{x^3}=(L_2)^{x^2}=(L_1)^x.
\]
We conclude that $L_3=L_2^x$ is different from $L_1$ and $L_2$. Of course $(L_3)^x=(L_1)^{x^3}=L_1$ and so $x$ cyclically permutes $L_1$, $L_2$ and $L_3$. 
\end{proof}

We treat the third case first.

\begin{prop} \label{tre un sesto ciclo}
Let $p$ be a prime, $H$ be a finite group and $G=(H_1 \times \dots \times H_p)\langle x \rangle$, with $H_i \simeq H$ for $i= 1,\dots,p$ and $x$ an element of order $p$ such that $H_i^x=H_{i+1}$ for $i=1, \dots, p-1$ and $H_p^x=H_1$. Then
\[
\spr_G(x) \leq \frac{1}{n_p(H)^{p-1}}.
\]
\end{prop}
\begin{proof}
Let $K=H_1 \times \dots \times H_p$. The Sylow $p$-subgroups of $G$ containing $x$ are exactly those of the form $Q \langle x \rangle$, where $Q$ is a Sylow $p$-subgroup of $K$ normalized by $x$. Moreover $Q$ is uniquely determined as $P \cap K$, where $P$ is the chosen Sylow $p$-subgroup containing $x$. Thus
\[
\lambda_G(x)=|\lbrace Q \in Syl_p(K) \ | \ Q^x=Q \rbrace |.
\]
A Sylow $p$-subgroup $Q$ of $K$ is the direct product of the Sylow $p$-subgroups $Q_i$ of $H_i$. If $P$ is $\langle x \rangle$-invariant, then $Q_i^x=Q_{i+1}$ for $i=1, \dots, p-1$ and $Q_p^x=Q_1$, so that the number of $\langle x \rangle$-invariant Sylow $p$-subgroups of $K$ is less or equal to the number of Sylow $p$-subgroups of $H$.

Finally,
\begin{equation}
\begin{split}
\spr_G(x)= & \frac{\lambda_G(x)}{n_p(G)}=\frac{|\lbrace Q \in Syl_p(K) \ | \ Q^x=Q \rbrace |}{n_p(G)} \\
\leq  & \frac{n_p(H)}{n_p(K)}= \frac{n_p(H)}{n_p(H)^p}=\frac{1}{n_p(H)^{p-1}}.
\end{split}
\end{equation}
\end{proof}

We can now apply this fact to case $3$ in Lemma \ref{riduzione tre un sesto CFSG}. Since $L$ is a simple group, $n_3(L) > 4$ and so
\[
\spr_G(x) \leq \frac{1}{n_3(L)^2} < \frac{1}{16}.
\]

By what we said as far as now, the proof of Proposition \ref{tre un sesto} is reduced to proving
\begin{prop} \label{tre un sesto almost}
Let $G$ a almost simple group, $L \leq G \leq Aut(L)$, with $L$ a nonabelian simple group, and let $x \in G$ be an element of order $3$. Then $\spr_G(x) \leq 1/6$.
\end{prop}

We now prove Proposition \ref{tre un sesto almost} for every family of simple groups.

\subsection{Alternating groups}
Here we consider $L=A_n$, where $5 \leq n \in \mathbb{N}$. This is the easiest case, but it is representative of the strategy used in all the others. In order to use Lemma \ref{sp x cresce sui sottogruppi}, we want to find a subgroup $H$ of $G$ such that $x \in H$ and $\spr_H(x)\leq 1/6$.

A direct calculation shows that in $G=A_5$ or $G=A_5 \times C_3$ every noncentral element $y$ of order $3$ is such that $\spr_G(y)=1/10$.

Let $n \geq 6$. Each element of order $3$ is product of, say $k$, $3$-cycles with pairwise disjoint supports. Without loss of generality we can assume that 
\[
x=(1,2,3) (4,5,6) \dots (3k-2,3k-1,3k),
\]
with $3k \leq  n$.
Put $x_1=(1,2,3)(4,5,6)$, $x_2=xx_1^2$ and $y=(3,5)(4,6)$.
Clearly $x_2 \in C_G(x_1) \cap C_G(y)$. It follows that 
\[
H=\langle x_1,x_2,y \rangle \simeq \langle x_1,y \rangle \times \langle x_2 \rangle \simeq A_5 \times C_3.
\]
By Lemma \ref{sp x cresce sui sottogruppi} we have that
\[
\spr_{A_n}(x) \leq \spr_{H}(x) =1/10.
\]
The almost simple case here is trivial, since $3$ does not divide $|Out(A_n)|$ for any $n$.
Thus Proposition \ref{tre un sesto almost} holds when $L \simeq A_n, \ n \geq 5$.

\subsection{Finite groups of Lie type: the \emph{fixed point ratio}}

We know that for a $p$-element $x$ of a finite group $G$, 
\begin{equation} \label{probabilita sp scritture}
\spr_G(x)= \frac{\lambda_G(x)}{n_p(G)},
\end{equation}
that is the ratio between the number of Sylow $p$-subgroups containing $x$ and the total number of Sylow $p$-subgroups of $G$. If we consider the transitive action of $G$ on its Sylow $p$-subgroups, we see that $\lambda_G(x)$ is the number of Sylow $p$-subgroups fixed (normalized) by $x$. Therefore, $\spr_G(x)$ is what is commonly called the \textit{fixed point ratio} ($\fpr$) of $x$ with respect to this action.

A first remark to be made is that an equality similar to (\ref{probabilita sp scritture}) holds even when the action we consider is not the one on the Sylow $p$-subgroups. Namely, if the action of $G$ on a set $\Lambda$ is transitive and $H$ is a point stabilizer, then
\[
\fpr_\Lambda(g)=\frac{|x^G \cap H|}{|x^G|}.
\]  

The fixed point ratio has been much studied (see for example \cite{burness:fprclassicalII}, \cite{liebeckshalev:simplegpspermgps}, \cite{gluckmagaard:charfprcclas} and the survey by Burness \cite{burness:surveyfpr}), especially for what concerns primitive actions of finite groups of Lie type.

The most general result in this area, which is very useful for our particular and circumscribed case, is the main theorem in \cite{liebecksaxl:mindegprimperm}.

\begin{teo}[Theorem 1 in \cite{liebecksaxl:mindegprimperm}] \label{liebsaxl} Let $L$ be a finite simple group of Lie type on $\mathbb{F}_q$, $q=p^f$, and let $G$ be an almost simple group with socle $L$ acting faithfully and primitively on a set $\Omega$. Then for all $1 \neq g \in G$ 
\begin{equation} 
\fpr_G(x) = \dfrac{|\fix_\Omega (g)|}{|\Omega |} \leq \dfrac{4}{3q},
\end{equation}
apart from a short list of known exceptions.
\end{teo}

Let us observe that the result is true for transitive actions as well, since if an action of $G$ on $\Lambda$ is transitive, then 
\[
\fpr_\Lambda(x)=\frac{|x^G \cap H|}{|x^G|} \leq \frac{|x^G \cap K|}{|x^G|} = \fpr_{G/K}(x),
\]
where $K$ is any maximal subgroup of $G$ containing $H$ and $G/K$ is the set of cosets of $K$ in $G$, on which $G$ acts primitively.

Consequently, when $G$ is an almost simple group of Lie type on $\mathbb{F}_q$ with $q \geq 8$, Proposition \ref{tre un sesto almost} is immediate. There are even better bounds for the exceptional groups of Lie type, studied in \cite{lawther:fprexceptional}, which give fixed point ratios smaller than $1/6$ even for values of $q$ smaller than $8$.

Thus, in completing the proof of Proposition \ref{tre un sesto almost}, we only have to deal with classical groups on $\mathbb{F}_q$ with $q<8$, besides the exceptions mentioned in Theorem \ref{liebsaxl}, which can be worked out with GAP (\cite{GAP4}), and the sporadic groups.

The following easy observation, whose proof is straightforward, allows us to work with nonprojective classical groups.

\begin{lem}
Let $G$ be a group, $p$ a prime, $x \in G$ a $p$-element and $Z=Z(G)$. Then $\spr_G(x)=\spr_{G/Z}(xZ)$. 
\end{lem}

We distinguish between two cases, whether or not $q=3$.

\subsection{Finite groups of Lie type: $q \neq 3$}

We begin our analysis with the linear case.

\begin{prop} \label{riduzione PSL 2 3}
Let $q=p^r$, with $p \neq 3$ a prime. Let $x \in GL_n(q)$, be a noncentral element such that $x^3 \in Z(GL_n(q))$. Then there is a noncentral $3$-element $y \in GL_m(q)$, with $m \in \lbrace 2,3 \rbrace$, such that
\[
\spr_{SL_n(q)\langle x \rangle}(x) \leq \spr_{SL_m(q) \langle y \rangle}(y).
\]
\end{prop}
\begin{proof}
We can assume that $x$ has order a power of $3$. Since $x^3 \in Z(GL_n(q))$, $x^3 = \omega I$, with $\omega$ a root of unity. 

Let then $V=\mathbb{F}_q^n$. Since $x \notin Z(G)$, we can take $v_1 \in V$ which is not an eigenvector for $x$.

We then take $v_2=v_1^x$ and $v_3=v_2^x$. Let $W$ be the span of $\lbrace v_1,v_2,v_3 \rbrace$. 
Since
\[
v_3^x=v_1^{x^3}=\omega v_1
\]
it is clear that $W$ is an $\langle x \rangle$-invariant subspace of dimension $2$ or $3$. Since $p \neq 3$, by Maschke's theorem there exists an $\langle x \rangle$-invariant complement $U$ of $W$.

Then $x \in A=GL(W) \times GL(U)$, say $x=x_1x_2$ in this decomposition. Clearly, for our choice of $W$, $x_1$ does not centralize $SL(W)$. 

Let $N=A \cap SL(V)$. We have
\begin{equation*}
\begin{split}
\spr_{SL(V)\langle x \rangle}(x) & \leq \spr_{N\langle x \rangle}(x) \leq \spr_{(SL(W) \times SL(U))\langle x \rangle}(x) \\
& \leq \spr_{SL(W)\langle x \rangle}(x)= \spr_{SL(W)\langle x_1 \rangle}(x_1).
\end{split}
\end{equation*}

All the inequalities follow from Lemma \ref{sp x cresce sui sottogruppi}, while last equality follows from the fact that $x_2$ centralizes $SL(W)\langle x_1 \rangle$.
\end{proof}

We now prove Proposition \ref{tre un sesto almost} when $L \simeq PSL_n(q)$.

If $q \in \lbrace 4,5,7 \rbrace$, we can check with GAP (\cite{GAP4}) that if 
\[
G=SL_m(q) \langle y \rangle
\]
with $m \in \lbrace 2,3 \rbrace$ and $|y|=3$, every element $g$ of order $3$ of $G$ is such that $\spr_G(g)\leq 1/6$. 

When $q=2$ we need a bigger $m$, since $SL_2(2)$ and $SL_3(2)$ are solvable. As in the proof of Proposition \ref{riduzione PSL 2 3} we find an $\langle x \rangle$-invariant subspace $W_1$ of dimension $2$ or $3$. We then take an $\langle x \rangle$-invariant complement $U$ of $W$. If $U$ is not an eigenspace for $x$, we repeat the argument to find an $\langle x \rangle$-invariant subspace $W_2$ of $U$ of dimension $2$ or $3$. Instead, if $x$ acts as a multiple of the identity on $U$ we just take $W_2$ to be any subspace of dimension $2$ of $U$. We set $W=W_1+W_2$ and repeat the conclusion of the proof of Proposition \ref{riduzione PSL 2 3}. Thus we have to check $3$-elements in groups of the form $PSL_m(2)\langle y \rangle$ with $4 \leq m \leq 6$ and $y$ a $3$-element. Again using GAP (\cite{GAP4}) we see that these elements are such that $\spr_G(g) \leq 1/6$.

This gives Proposition \ref{tre un sesto almost} for $L \simeq PSL_n(q)$, since any automorphism of order $3$ of $PSL_n(q), q<8,$ comes from conjugation by elements of $GL_n(q)$.

Now we look at classical groups with forms, i.e., symplectic, unitary and orthogonal groups. The next lemma is our main tool in inspecting these groups.

\begin{lem} \label{non singolare}
Let $V$ be a vector space over $\mathbb{F}_q$ of dimension $n$ and $f$ a nondegenerate alternating, sesquilinear or symmetric bilinear form on $V$. Let $x$ be an endomorphism of $V$ such that for all $v,w \in V$, $f(v,w)=0$ if and only if $f(v^x,w^x)=0$. Moreover let $W$ be an $\langle x \rangle$-invariant subspace of dimension $d$. Then there is a nonsingular $\langle x \rangle$-invariant subspace of dimension $k$, with $d \leq k \leq 2d$.
\end{lem}

\begin{proof}
The subspace $\tilde{W}=W+W^\perp$ is $\langle x \rangle$-invariant. 

Let $U$ be an $\langle x \rangle$-invariant complement of $\tilde{W}$ in $V$, whose existence is guaranteed by Maschke's theorem. We observe that 
\begin{equation*}
\begin{split}
\dim (U) &= \dim (V) - \dim (\tilde{W}) = \\ &=\dim (V) - (\dim(W)+\dim(W^\perp)-\dim(W \cap W^\perp))= \\ &= \dim (W \cap W^\perp)
\end{split}
\end{equation*}
We now consider $Y:=W \oplus U$. Note that $Y$ is $\langle x \rangle$-invariant and of dimension $k$, with $d \leq k \leq 2d$.  We show that $Y$ is nonsingular. Let $W=W_0 \oplus W_1$ with $W_0=W \cap W^\perp$. 

We have $Y^\perp=W^\perp \cap U^\perp$ and $V=W^\perp \oplus (W_1 \oplus U)$. 

Let $v \in Y \cap Y^\perp$. Then $v=w_0+w_1+u$ with $w_0 \in W_0$, $w_1 \in W_1$ and $u \in U$. Since $v$ and $w_0$ belong to $W^\perp$, we have $w_1+u \in W^\perp \cup (W_1 \oplus U)$ and so $w_1+u=0$. Then $v=w_0 \in W \cap W^\perp \cap U^\perp$, and so
\[
\langle v \rangle^\perp \geq W + W^\perp + U=V 
\]
which implies $v=0$, that is $Y$ is nonsingular.
\end{proof}

We can now prove Proposition \ref{tre un sesto almost} when $L$ is a symplectic, unitary or orthogonal group on $\mathbb{F}_q$, $q<8$. In this case an element $x$ of order $3$ in $Aut(L)$ is inside $L$ unless $L=PSU_n(q)$ and $x$ comes from an element in $GU_n(q)$, or $L=P\Omega_8^+(q)$. Suppose we are not in the latter case, which will be treated later.

Let $\Delta$ be one of the following groups: 
\[
GU_n(q),\ Sp_n(q),\ O_n^\pm(q). 
\]
As described in \cite{kleidmanliebeck:sbgpstrctclas}, the stabilizer in $\Delta$ of a decomposition 
\begin{equation} \label{dec ort}
V=W \perp U
\end{equation}
is the direct product $\Delta(W) \times \Delta(U)$ (the sign of the orthogonal group can change). Moreover, let $x$ be a $3$-element in $\Delta$, and $G=S\Delta \langle x \rangle$ where $S\Delta$ is the quasi-simple group contained in $\Delta$. If $x$ stabilizes a decomposition (\ref{dec ort}), without centralizing $W$, we have that $x=x_1x_2 \in \Delta(W) \times \Delta(U)$, $S\Delta(W)$ is stabilized by $x,x_1$ and $x_1$ does not centralize it. Then
\[
\spr_G(x) \leq \spr_{S\Delta(W)\langle x_1 \rangle} (x_1).
\] 

Our strategy will be to find a nonsingular subspace of small dimension $W$, not centralized by $x$, such that $PS\Delta(W)$ is not solvable and then to check the finitely many resulting groups with GAP (\cite{GAP4}).

Let $L \simeq PSU_n(q)$. We prove the $q \neq 2$ case first. As in the proof of Proposition \ref{riduzione PSL 2 3}, we find an invariant subspace $W$ of dimension $2$ or $3$ on which the action of $x$ is not scalar. By Lemma \ref{non singolare} we find an $\langle x \rangle$-invariant nonsingular subspace $Y$ of dimension $2 \leq d \leq 6$. 
Then we only need to check the groups $PSU_d(q)$, with $2 \leq d \leq 6$, and their extensions by a $3$-element. 
We use some bounds on the size of conjugacy classes of semisimple elements contained in \cite{burness:fprclassicalII}.
Namely, it is enough to use the inequality 
\begin{equation} \label{sp lasca}
\spr_G(x) \leq \frac{|P|}{|x^G|}
\end{equation}
in order to get ratios smaller than $1/6$.

Let $q=2$. Again, using Lemma \ref{non singolare}, we take an $\langle x \rangle$-invariant nonsingular subspace $Y$ of dimension $2 \leq d \leq 6$, on which the action of $x$ is not scalar. If $4 \leq d \leq 6$ we can check with GAP that Proposition \ref{tre un sesto almost} holds for $L=PSU_d(2)$. If instead $d \in \lbrace 2,3 \rbrace$, we have to consider a subspace of greater dimension, since $PSU_d(2)$ is solvable. We have two cases, whether $Y$ is or is not an eigenspace for $x$. \\
If not, we can repeat the argument of Lemma \ref{non singolare} to find a nonsingular subspace $U \leq Y^\perp$ of dimension $2 \leq d' \leq 6$, on which $x$ does not acts as a scalar. If $4 \leq d' \leq 6$ we can replace $Y$ with $U$. If on the contrary $d' \in \lbrace 2,3 \rbrace$ we consider the nonsingular subspace $Y+U$, whose dimension is between $4$ and $6$.\\
This proves Proposition \ref{tre un sesto almost} when $L \simeq PSU_n(q)$.

When $L \simeq PSp_n(q)$, the argument is similar to the one used for $PSU_n(q)$.

If $L \simeq P\Omega_n (q) $, with $q \in \lbrace 5,7 \rbrace$, the argument used for $PSU$ works, since the quadratic form defining $L$ is equivalent to the related bilinear form.\\
In characteristic $2$ we are only interested with $P\Omega_{2n}^\pm$, since $P\Omega_{2n+1}(2^r) \simeq PSp_{2n}(2^r)$.\\
When the dimension is even, calling $Q$ the quadratic form and $f$ the related bilinear symmetric form, we have that $Rad(f)=0$ if and only if $Rad(Q)=0$.\\
Thus Lemma \ref{non singolare} still gives a subspace $Y$, nonsingular with respect to the bilinear form and so also to the quadratic form (since $Rad(Q) \subseteq Rad(f)$). This allows to threat the inner and diagonal elements of the orthogonal groups as we did for the unitary and the symplectic case. \\
The only elements to be checked are hence the outer automorphisms of $L=P\Omega_8^+(q)$ with $q \leq 7$ (this is the only case having graph automorphisms involved).\\ 
In \cite{burness:fprclassicalII}, Lemma 3.48, the following bound is given for a $3$-element $x \in Aut(L) \setminus L$: 
\[
|x^L| \geq \frac{1}{8} q^{14}.
\]
Set $G=L\langle x \rangle$. For $q \in \lbrace 2,4,5,7 \rbrace$ the $3$-Sylow subgroup $P$ in $P\Omega_8^+(q)$ has order $243$ and so
\[
\spr_G(x) = \frac{|x^G \cap P\langle x \rangle|}{|x^G|} \leq \frac{3|P|}{|x^L|} \leq \frac{8 \cdot 729}{q^{14}}
\]
which is less than $1/6$ for $q \geq 4$.\\
As for $P\Omega_8^+(2)$, we check with GAP (\cite{GAP4}).

\subsection{Finite groups of Lie type: $q = 3$}

Again the case of an outer automorphism of $P\Omega_8^+(3)$ can be checked with GAP (\cite{GAP4}). We thus only have to deal with elements inside $L$. 

In \cite{gonshaw:unipclas}, a description of unipotent classes in classical groups is given in terms of Jordan blocks. Since the element $x$ has order $3$, its Jordan blocks have dimensions between $1$ and $3$. 

In the linear case it is enough to take the subspace related to a Jordan block of dimension $3$, two Jordan blocks of order $2$, or one of order $1$ and the other of order $2$. We can use GAP to check $PSL_3(3)$ and $PSL_4(3)$.

In the other cases call $J_i$ a Jordan block of dimension $i \in \lbrace 1,2,3 \rbrace$, and if $r_i$ is the multiplicity of that block we write the Jordan form of $x$ as
\begin{equation}\label{jordan}
r_1J_1+r_2J_2+r_3J_3.
\end{equation}

Following the description of \cite{gonshaw:unipclas}, we describe what happens when $L=PSU_n(3)$. Proposition 2.2 in \cite{gonshaw:unipclas} tells that for all $i \in \lbrace 1,2,3 \rbrace $ there exist $r_i$ subspaces of $V$ of dimension $i$ which are pairwise orthogonal, $\langle x \rangle$-invariant and such that the sesquilinear form is nonsingular on all of them. Choosing a sum of these subspaces on which $x$ acts nontrivially, we get a nonsingular subspace $W$ of dimension $d \in \lbrace 3,4 \rbrace$ such that 
$x|_W \neq id_W$. $W^\perp$ is $\langle x \rangle$-invariant too and so 
\[
\spr_L(x) \leq \spr_{SU_d(3)}(y)
\]
for some $3$-element $y$.

In the symplectic and orthogonal case things go likewise, but $\dim (W) \in \lbrace 4,5,6 \rbrace$.

\subsection{Sporadic groups}

Looking at the ATLAS of finite simple groups, it comes out that if $L$ is each one of the twentysix sporadic simple groups, $P \in Syl_3(L)$, it is true the loose inequality
\[
\spr_L(x) \leq \frac{|P|}{|x^L|} \leq \frac{1}{6}.
\]
The almost simple case is trivial with sporadic groups, since for such an $L$ we have $|Out(L)| \in \lbrace 1,2 \rbrace$.

This completes the proof of Proposition \ref{tre un sesto almost}, which, together with Lemma \ref{riduzione tre un sesto CFSG} and Proposition \ref{tre un sesto ciclo}, prove Proposition \ref{tre un sesto} as well.	

\section{The number of $2$-elements in nonsolvable groups} \label{section frob}

In this section we focus on \textit{Step 2} of the proof of Theorem \ref{unsesto}, as described at the end of section \ref{sketch}. 

We now prove the fact that the sum of $\spr_G(x)$ over all the $p$-elements of $G$ gives the cardinality of a Sylow $p$-subgroup.

\begin{lem} \label{somma dei subnormalizzatori dei p-elementi}
Let $p$ be a prime dividing the order of a finite group $G$ and $P \in Syl_p(G)$. Moreover let $\mathfrak{U}_p(G)$ be the set of the $p$-elements in $G$. Then
\[
\sum_{x \in \mathfrak{U}_p(G)} |S_G(x)| =|P||G|.
\]
\end{lem}
\begin{proof}
Let $\mathcal{K}_1, \dots \mathcal{K}_n$ be the conjugacy classes of $p$-elements in $G$ and, for every $i$, choose an element $x_i \in P \cap \mathcal{K}_i$. Then by Theorem \ref{form subnor theorem} and Remark \ref{ossalfa}
\begin{align*}
\sum_{x \in \mathfrak{U}_p(G)} |S_G(x)| & = \sum_{i=1}^n |K_i| \alpha_G(x_i)|C_G(\langle x_i \rangle)| \\
&= \sum_{i=1}^n [G:C_G(x_i)] = |G| \sum_{i=1}^n \alpha_G(x_i) = |G||P|.
\end{align*}
The last equality holds since every element in $P$ is conjugate to one of the elements $x_i$.
\end{proof}

As shown in the outline, this allow us to deal with the ratio $|\mathfrak{U}_2(G)|/|G|_2$ for nonsolvable groups. This ratio, when $2$ is replaced by any prime $p$, is studied more in general in \cite{gheri:pelem}.

We collect some information about $\mathfrak{U}_p(G)$ in the following lemma, whose proof is straightforward.

\begin{lem} \label{Frob bischerate}
Let $P$ be a Sylow $p$-subgroup of $G$. 
\begin{itemize}
\item[i)] If $P \leq H \leq G$, then 
\[
\frac{|\mathfrak{U}_2(G)|}{|G|_2} \geq \frac{|\mathfrak{U}_2(H)|}{|H|_2}.
\]
\item[ii)] If $N$ is a normal subgroup of $G$, then
\[
|\mathfrak{U}_p\left( G/N \right)|/|(G/N)|_p \leq |\mathfrak{U}_p(G)|/|G|_p.
\]
Also, if $N \leq Z(G)$, then equality occurs.
\item[iii)] If $H$ is a subgroup normalized by a $p$-element $g$ and $Q$ is a Sylow $p$-subgroup of $H$, then
\[
|\mathfrak{U}_p(Hg)| \geq |Q|.
\]
\end{itemize}
\end{lem}

The next lemma will be useful to deal with groups in $\mathfrak{M}_{ns}$. 

\begin{lem} \label{Omega in coset prodotto diretto}
If $N = N_1 \times \dots \times N_t \unlhd G$ and $g \in \mathfrak{U}_p(G)$ is an element such that $g \in N_G(N_i)$, $\forall i \in \lbrace 1, \dots, t \rbrace$, we have
\[
\left| \mathfrak{U}_p(Ng)\right| = \prod_{i=1}^t|\mathfrak{U}_p(N_ig)|.
\]
In particular, 
\[
\left| \mathfrak{U}_p(N)\right| = \prod_{i=1}^t|\mathfrak{U}_p(N_i)|.
\]
\end{lem}
\begin{proof}
It is enough to prove the lemma for $t=2$, as the general case follows from it with an easy induction.
Arguing by induction on $s \in \mathbb{N}$ we see that
\[
(xy)^s=xx^{y^{-1}}x^{y^{-2}} \cdots x^{y^{-(s-1)}}y^s,
\]
for all $x,y \in G$.\\
Let $|g|=p^r$ and $a \in N_1,b \in N_2$ such that $|(a,b)g|=p^l$. If $l \geq r$
\begin{equation*}
\begin{split}
&1=((a,b)g)^{p^l} =(a,b)(a,b)^{g^{-1}}(a,b)^{g^{-2}} \cdots (a,b)^{g^{-(p^l-1)}}g^{p^l}=\\
&=(aa^{g^{-1}} \cdots a^{g^{-(p^l-1)}},bb^{g^{-1}} \cdots b^{g^{-(p^l-1)}}).
\end{split}
\end{equation*}
and so 
\begin{equation*}
\begin{split}
aa^{g^{-1}} \cdots a^{g^{-(p^l-1)}}=1 \\
bb^{g^{-1}} \cdots b^{g^{-(p^l-1)}}=1,
\end{split}
\end{equation*}
that is, $|ag|$ and $|bg|$ divide $p^l$.

If on the contrary $l<r$, we have
\begin{equation*}
1=\left( ((a,b)g)^{p^l} \right)^{p^{r-l}} = \left( aa^{g^{-1}} \cdots a^{g^{-(p^{r-1})}},bb^{g^{-1}} \cdots b^{g^{-(p^{r-1})}} \right),
\end{equation*}
and so again $|ag|$ and $|bg|$ divide $p^l$.

The following map is then well defined and injective
\begin{align*}
\Phi: \mathfrak{U}_p((N_1 \times N_2)g) & \rightarrow \mathfrak{U}_p(N_1g) \times \mathfrak{U}_p(N_2g), \\ 
(a,b)g & \mapsto (ag,bg).
\end{align*}

As for the surjectivity, we observe that if $|ag|=p^{l_1}$ and $|bg|=p^{l_2}$, then, setting $l=\max \lbrace r,l_1,l_2 \rbrace$, we have
\[
((a,b)g)^{p^l}=(aa^{g^{-1}} \cdots a^{g^{-(p^l-1)}},bb^{g^{-1}} \cdots b^{g^{-(p^l-1)}})=1.
\]
\end{proof}	

Now let $G$ be a group in $\mathfrak{M}_{ns}$ and let $N$ be the unique minimal normal subgroup of $G$. Then $N$ is the direct product
\[
L_1 \times \dots \times L_k
\] 
of $k$ copies of the same nonabelian simple group $L$ (so each $L_i \simeq L$). Moreover $C_G(N)=1$, $G \lesssim Aut(N) \left( \simeq Aut(L) \wr S_k \right)$.

We shall prove that $|\mathfrak{U}_2(G)|/|G|_2  \geq 6$.

Using part (i) of Lemma \ref{Frob bischerate}, it is enough to prove Proposition \ref{due un sesto} in the case $G=NP$ with $P$ a Sylow $2$-subgroup of $G$.

As we said before, our arguments fail for a finite number of groups, which we treat separately in the following lemma.

\begin{lem} \label{frob A5}
Let $G \in \mathfrak{M}_{ns}$ and $N$ be its minimal normal subgroup. If $N$ is isomorphic to one of the following groups
\[
A_5, A_5 \times A_5, PSL_2(7),PSL_2(16)
\] 
then $\spr(G) \leq 1/6$.
\end{lem}
\begin{proof}
Since $C_G(N)=1$, we have that $G$ is isomorphic to a subgroup of $Aut(N)$. If $N$ is simple, then $G$ is an almost simple group with socle $N$, and these cases can be checked by direct calculation with GAP (\cite{GAP4}). 

If $N \simeq A_5 \times A_5$, then, since $C_G(N)=1$, $G$ is isomorphic to a subgroup of $Aut(N)$, which is an extension of $N$ of index at most $8$. Again by calculation one can see that the ratio $|\mathfrak{U}_2(G)|/|G|_2$ of any such group is greater than $6$.
\end{proof}

To be clearer we can now give the exact statement that is proved in the rest of the section.

\begin{prop} \label{due un sesto}
Let $G \in \mathfrak{M}_{ns}$ and $N$ be its minimal normal subgroup. If $N$ is not isomorphic to one of the groups in Lemma \ref{frob A5}, then 
\[
|\mathfrak{U}_2(G)|/|G|_2 \geq 6.
\]
\end{prop}

A group $G \in \mathfrak{M}_{ns}$ can be embedded in the wreath product $Aut(L) \wr S_k$. Let $B$ be the base of this wreath product, 
\[
B=Aut(L_1) \times \dots \times Aut(L_k).
\]
Set $K=B \cap G \unlhd G$, $\mathcal{T}_K$ a right transversal of $P \cap N$ in $P \cap K$ and $\mathcal{T}_G$ a right transversal of $P \cap N$ in $P$ such that $1 \in \mathcal{T}_{K} \subseteq \mathcal{T}_{G}$.

These are right transversals of $N$ in $K$ and $G$, respectively, and so
\[
G= K \ \dot{\cup} \dot{\bigcup_{g \in \mathcal{T}_{G} \setminus \mathcal{T}_{K}}} \left( Ng \right).
\]
Observing that 
\[
|P|=\frac{\left| P \right|}{\left| P \cap N \right| }\left|P \cap N \right| = \frac{\left| PN \right|}{\left| N \right| }\left|P \cap N \right|=\left| \frac{G}{N} \right| \left|P \cap N \right|
\]
and $|P|=\left| G/K \right| \left|P \cap K \right|$ in the same way, we get
\begin{equation} \label{stimaschifo}
\dfrac{|\mathfrak{U}_2(G)|}{|P|}  = \dfrac{|\mathfrak{U}_2(K)|}{|P|} + \dfrac{1}{|P|} \sum_{g \in \mathcal{T}_{G} \setminus \mathcal{T}_{K}} |\mathfrak{U}_2(Ng)| =  \mathcal{A} + \mathcal{B},
\end{equation}
where we set
\begin{equation} \label{A e B}
\begin{split}
\mathcal{A}=& \dfrac{|\mathfrak{U}_2(K)|}{\left| \frac{ G}{K }\right|  \left| P \cap K \right| }\\
\mathcal{B}=& \dfrac{1}{\left| \frac{G}{N}\right| \left|P \cap N \right| } \left( \sum_{g \in \mathcal{T}_{G} \setminus \mathcal{T}_{K}} \left| \mathfrak{U}_2(Ng) \right| \right).
\end{split}
\end{equation}

We now treat the two terms $\mathcal{A}$ and $\mathcal{B}$ separately. 

\subsection{A bound for $\mathcal{B}$}

First of all we observe that up to conjugation by an element of $Aut(N)$, we can take $P$ inside $Q \wr R$, the wreath product of a $Q \in Syl_2(Aut(L))$ and $R \in Syl_2(S_k)$. Set $P_0 = Q \cap L$, a Sylow $2$-subgroup of $L$.

The next proposition, and the following corollary, give a bound for any summand of $\mathcal{B}$, that is, for the number of $2$-elements in a coset of $N$ not contained in $K$.

\begin{prop} \label{coset intrecciato ciclo}
Suppose that $k=2^t$ and let $\sigma=(1,2,\dots,k) \in S_k$. Let, moreover, $v=\left( v_1, \dots, v_k \right) \in B$ such that $v\sigma \in P$. Then
\[
|\mathfrak{U}_2(Nv\sigma)| \geq \frac{|L|^k}{|C_L(au)|},
\]
where $a$ is any element of $P_0$ and $u=v_1v_2 \dots v_{k} \in Aut(L)$.
\end{prop}
\begin{proof}
Let $a \in P_0$ and $\tilde{a}=(a,1, \dots ,1) \in N$. Since the direct product of $k$ copies of $P_0$ lies in $P$, we have that the element $g=\tilde{a}v\sigma \in P$. Moreover, for all $x \in N$ we have
\[
g^x=\tilde{a}^x(v\sigma)^x=\tilde{a}^x x^{-1}x^{(v\sigma)^{-1}}v\sigma \in Nv\sigma,
\]
and so $g^N \subseteq \mathfrak{U}_2(Nv\sigma)$.

The size of the $N$-orbit of $g$ is given by
\[
|g^N|=\dfrac{|N|}{|C_N(g)|}.
\]
Let $x = (x_1, \dots, x_k) \in N$. Then $x \in C_N(g)$ if and only if $x^{\tilde{a}v\sigma}=x,$
that is, if an only if
\[
(x_k^{v_k},x_1^{av_1}, \dots, x_{k-1}^{v_{k-1}})=(x_1, \dots, x_k),
\]
or, equivalently:
\[
x_1=x_k^{v_k}=(x_{k-1}^{v_{k-1}})^{v_k} = \dots = x_1^{av_1 \dots v_k}=x_1^{au}.
\]
For every choice of $x_1 \in C_L(au)$ the other components of $x$ are uniquely determined and so
\[
\left| C_N(g) \right| = \left| C_L(au) \right|.
\]
\end{proof}

\begin{cor} \label{coset intrecciato permutazione}
Let $1 \neq \sigma \in R$ and $v=(v_1, \dots , v_k) \in Q^k$ so that $v\sigma \in Q \wr R$. Let $s$ be the maximal size of an orbit of $\sigma$. Then 
\[
|\mathfrak{U}_2(Nv\sigma)| \geq \frac{|L|^s}{|C_L(g)|} \left| P_0 \right|^{k-s},
\]
where $g \in Aut(L)$. 
\end{cor}
\begin{proof}
Let $\sigma=\sigma_1 \dots \sigma_m$ be the expression of $\sigma$ as a product of disjoint cycles (including those of size $1$) with decreasing sizes (so that $\sigma_1$ has size $s$) and set
\begin{equation*}
\mathcal{O}_i=Supp(\sigma_i), \ \ N_i=\prod_{j \in \mathcal{O}_i} L_j,\ \ \tau= \prod_{i>1} \sigma_i. 
\end{equation*}
Moreover, through the identification 
\[
Aut(L_j) \simeq 1 \times \dots \times Aut(L_j) \times \dots \times 1
\]
we write $v$ as $v_1 \dots v_k$ and set
\begin{equation*}
\begin{split}
w_i=\prod_{j \in \mathcal{O}_i} v_j, \\
w=\prod_{i>1} w_i.
\end{split}
\end{equation*}
Then $N$ is the direct product of the subgroups $N_i$ and $v\sigma$ normalizes every $N_i$. We can then apply Lemma \ref{Omega in coset prodotto diretto} and obtain
\begin{equation} \label{Ni vu sigma}
\left| \mathfrak{U}_2(Nv\sigma)\right| = \prod_{i=1}^m|\mathfrak{U}_2(N_iv\sigma)|.
\end{equation}
Since $\langle N_1,w_1,\sigma_1 \rangle$ and $\langle w,\tau \rangle$ commute, we have 
\[
v\sigma=w_1w\sigma_1 \tau=(w_1\sigma_1) (w\tau).
\]
Moreover, an element $a w_1 \sigma_1 w \tau \in N_1 w_1 \sigma_1 w \tau$ is a $2$-element if and only if $a(w_1\sigma_1)$ is such.
We can then apply Lemma \ref{coset intrecciato ciclo} and get
\[
|\mathfrak{U}_2(N_1v\sigma)| \geq \frac{|L|^s}{|C_L(g)|},
\]
for some $g \in Aut(L)$.
For the other terms of the product (\ref{Ni vu sigma}) we use part (iii) in Lemma \ref{Frob bischerate}. Finally we get
\[
\left| \mathfrak{U}_2(Nv\sigma)\right| = \frac{|L|^s}{|C_L(g)|} \prod_{i=2}^m \left| P_0 \right|^{\left| \mathcal{O}_i \right|}=\frac{|L|^s}{|C_L(g)|} \left| P_0 \right|^{k-s}.
\]
\end{proof}

We can now obtain a useful bound for $\mathcal{B}$. First of all, $P \cap N$ is a Sylow $2$-subgroup of $N$ and so its cardinality is $|P_0|^k$. Every element $g \in \mathcal{T}_{G} \setminus \mathcal{T}_{K}$ is the product of an element $v \in Q^k$ by an element $1 \neq\sigma \in S_k$, so we can use Corollary \ref{coset intrecciato permutazione}. For such an element $g=v\sigma$, set $s_g$ the maximal size of an orbit of $\sigma$.

Setting
\[
c=\max_{x \in Aut(L)} |C_L(x)|,
\]
we get
\begin{equation*}
\begin{split}
\mathcal{B}=& \dfrac{1}{\left| \frac{G}{N}\right| \left| P \cap N \right| } 
\left( \sum_{g \in \mathcal{T}_{G} \setminus \mathcal{T}_{K}} \left| \mathfrak{U}_2(Ng) \right| \right) \\ 
\geq & \dfrac{1}{\left| \frac{G}{N}\right| \left| P_0 \right|^k } 
\left( \sum_{g \in \mathcal{T}_{G} \setminus \mathcal{T}_{K}}
\frac{\left| L \right|^{s_g}}{c} \left| P_0 \right|^{k-s_g}  \right)  \\
\geq & \dfrac{1}{\left| \frac{G}{N}\right| \left|P_0 \right|^k } \cdot \frac{1}{c} \cdot \left( \sum_{g \in \mathcal{T}_{G} \setminus \mathcal{T}_{K}} \left( \frac{|L|}{|P_0|} \right)^{s_g} \left| P_0 \right|^{k} \right)  \\
\geq & \dfrac{\left| \mathcal{T}_{G} \setminus \mathcal{T}_{K}\right|}{\left| \frac{G}{N}\right|} \cdot \frac{1}{c} \cdot \dfrac{\left| L \right|^2}{\left|P_0 \right|^k }|P_0|^{k-2} \\
= & \dfrac{\left| G \right| - \left| K \right|}{\left| G \right|} \cdot \frac{\left| L \right|}{c} \cdot 
\frac{\left| L \right|}{|P_0|^{2}} 
\end{split}
\end{equation*}

The following theorem, whose proof relies on CFSG, ensures that the third factor of the last product is greater than $1$.

\begin{teo}\cite{lyons:sylow} Let $L$ be a nonabelian finite simple group, $p$ a prime dividing $\left| L \right|$ and $P \in Syl_p(L)$. Then $\left| P \right|^2 < \left| L \right|$.
\end{teo}

As for the factor $|L|/c$, we observe that a nonabelian simple group has a proper subgroup of index smaller then $6$ if and only if it is isomorphic to $A_5$. Therefore, if $L \neq A_5$, $|L|/c \geq 6$.
For $L=A_5$, we can calculate $c$ directly to see that $|L|/c=10$. Finally we obtain the bound
\begin{equation} \label{bound for B}
\mathcal{B} \geq 6 \dfrac{\left| G \right| - \left| K \right|}{\left| G \right|}.
\end{equation}

\subsection{A bound for $\mathcal{A}$}

Since $N \leq K$, we have that $|\mathfrak{U}_2(K)| \geq |\mathfrak{U}_2(N)|= |\mathfrak{U}_2(L)|^k$. If $\hat{P}$ is a Sylow $2$-subgroup of $Aut(L)$,	 we get $|P \cap K| \leq |\hat{P}|^k$ and so
\begin{equation} \label{frob auto k}
\mathcal{A}= \dfrac{|\mathfrak{U}_2(K)|}{\left| \frac{ G}{K }\right|  \left| P \cap K \right| } \geq \dfrac{|\mathfrak{U}_2(L)|^k}{\left| \frac{ G}{K }\right| | \hat{P} |^k } = 
\dfrac{|K|}{|G|} \left( \dfrac{|\mathfrak{U}_2(L)|}{|\hat{P}|} \right)^k.
\end{equation}

We will show that this bound is enough for proving that $\mathcal{A} \geq 6|K|/|G|$, which, together with (\ref{bound for B}), will complete the proof of Proposition \ref{due un sesto}. For every $L$ nonabelian simple group, set
\[
\phi \left( L \right) = \frac{\left| \mathfrak{U}_2(L) \right|}{| \hat{P}|}.
\]
We will show that this ratio is strictly greater than 5 for every $L$ nonabelian simple group, except for the cases treated in Lemma \ref{frob A5}.

Again we consider the different cases separately.

\subsubsection{Alternating groups}

Let $L = A_n$, $n \geq 6$ and let $n=2^{m_1}+2^{m_2}+ \dots + 2^{m_l}$, with $m_1 > m_2 > \dots > m_l \geq 0$. Set $n_i=\sum_{j=1}^i 2^{m_j}$ for $i \in \lbrace 1, \dots, l \rbrace$. Take 
\[
s=( 1,2, \dots, n_1 ) \dots \big(n_{l-1}+1, \dots , n \big),
\]
if this is an element of $A_n$, or
\[
s=( 1,2, \dots, n_1 ) \dots \big(n_{l-1}+1, \dots , n_{l-1}+2^{m_l-1} ) (n_{l-1}+2^{m_l-1}+1, \dots n),
\]
otherwise.

The conjugacy class of $s$ in $S_n$ has at least
\[
|s^L| \geq \dfrac{n!}{2 \cdot 2^{m_1} \cdots 2^{2m_l}} = \dfrac{n!}{2^{(\sum_{i=1}^lm_i) + m_l+1}}
\]
elements.
Since $m_l \leq \left \lfloor \log_2(n) \right \rfloor$, we get that 
\[
\sum_{i=1}^l m_j \leq \sum_{i=1}^{\left \lfloor \log_2(n) \right \rfloor}i \leq \log_2^2(n)+\log_2(n),
\]
and so
\[
|s^L| \geq \dfrac{n!}{2^{\frac{1}{2}(\log_2^2(n)+3\log_2(n))+1}}.
\]
The order of a Sylow $2$-subgroup of $L$ is
\[
|P_0|=\frac{1}{2} 2^{2^{m_1}-1} \cdots 2^{2^{m_l}-1}=\frac{2^n}{2^{l+1}} \leq 2^{n-1}.
\]
Therefore, for $n>6$ we have
\[
\phi(L) > \frac{\left| s^L \right|}{2|P_0|} \geq \dfrac{n!}{2^{\frac{1}{2}(\log_2^2(n)+3\log_2(n))+n+1}}.
\]
The last term of this inequality is an increasing function in $n$. Since $\phi(A_{12})>6$, we have that $\phi(A_n) \geq 6$ for $n \geq 12$, while, for $6 \leq n < 12$, we can calculate $\phi(A_n)$ directly to get the desired bound. 

We observe that when $L \simeq A_5$, we have $\phi(L)=2$. The bound (\ref{frob auto k}) fails when $k=1,2$, and this is the reason why these cases were treated separately in (\ref{frob A5}).

\subsubsection{Groups of Lie type in odd characteristic}

As we have just seen, for alternating groups it is enough to bound $|\mathfrak{U}_2(L)|$ with the size of a single conjugacy class to get the desired bound. We do the same thing with groups of Lie type in odd characteristic.

We use the results on $2$-regular elements contained in \cite{guralnick:surjwordmaps}, where the authors find explicit bounds for the size of the conjugacy classes of elements (Lemma 7.5 and following). Namely, for every class of groups of Lie type in odd characteristic, the authors are able to find a regular $2$-element $s$ and to give an upper bound for the order of its centralizer. Clearly, this gives a lower bound for the order of the conjugacy class $s^L$.

We then use the following inequality:
\[
\phi(L) \geq \frac{\left| s^L \right|}{|Out(L)|_2 \cdot |L|_2}.
\]
The aforementioned results contained in \cite{guralnick:surjwordmaps}, together with some easy estimates on the $2$-parts of $|L|$ and $|Out(L)|$, give the desired bound for all but a finite number of simple groups, which can be checked directly using GAP.

\subsubsection{Groups of Lie type in characteristic 2}

A theorem of Steinberg gives an exact result in this case for the ratio we are interested in.

\begin{teo}(15.2 in \cite{steinberg:endom}).
Let $G$ be a connected reductive algebraic group over a field of characteristic $p$ and let $F$ be a Frobenius map. Then 
\[
|\mathfrak{U}_p(G^F)| = (|G^F|_p)^2.
\]
\end{teo} 

Since each finite quasisimple group of Lie type $L$ is the group of fixed points of a certain Frobenius map $F$ on a connected reductive group $G$, using Steinberg's theorem together with part (ii) of Lemma \ref{Frob bischerate}, we get
\[
\phi(L)=\dfrac{\left| \mathfrak{U}_2(L) \right|}{\left|Out(L) \right|_2 \left| P_0 \right| } = \dfrac{\left| \mathfrak{U}_2(G^F) \right|}{\left|Out(L) \right|_2 \left| G^F \right|_2 } = \dfrac{|G^F|_2}{\left| Out(L) \right|_2}.
\]

Some easy estimates on the order of a Sylow $2$-subgroups of the groups $G^F$ and on $|Out(L)|_2$ give the desired bound.

\subsubsection{Sporadic groups}

The CTblLib library of GAP (\cite{GAP4}) contains the conjugacy class sizes of every sporadic group. One can thus compute the exact value of $\phi(L)$ for $L$ sporadic and see that this is always greater than $6$. 

\subsection*{ }

If $N$ is not one of the groups in Lemma \ref{frob A5}, then we have
\[
\frac{|\mathfrak{U}_2(G)|}{|G|_2} \geq \mathcal{A} + \mathcal{B} \geq 6 \frac{|K|}{|G|} + 6 \frac{|G|-|K|}{|G|} =6
\]
This completes the proof of Proposition \ref{due un sesto} and so that of Theorem \ref{unsesto} too.

\subsection*{Acknowledgements} 

This article is part of the author’s PhD thesis which was written under the great supervision of Prof. Carlo Casolo, whose contribution to this work was essential. 

The work was also in part developed during a visit to Rutgers University which led to fortunate encounters with Robert Guralnick, Richard Lyons and Pham Huu Tiep, who all deserve thanks, together with Silvio Dolfi, who suggested and facilitated the visit.

Thanks are also due to Francesco Fumagalli for his valuable comments and suggestions.

This work was partially funded by the Istituto Nazionale di Alta Matematica ``Francesco Severi" (Indam).

\end{document}